\documentclass[8pt]{article}
\usepackage{indentfirst,latexsym,bm}
\usepackage{times}
\usepackage{amsfonts}
\usepackage{amssymb}
\usepackage[leqno]{amsmath}
\usepackage{dsfont}
\usepackage{amsthm}
\usepackage[all]{xy}
\usepackage{hyperref}
\usepackage{color,xcolor}
\begin{document}
\title{On the  center of near-group fusion category of type $\mathbb{Z}_3+6$}
\author{Zhiqiang Yu}
\date{}
\maketitle

\newtheorem{theo}{Theorem}[section]
\newtheorem{prop}[theo]{Proposition}
\newtheorem{lemm}[theo]{Lemma}
\newtheorem{coro}[theo]{Corollary}
\theoremstyle{definition}
\newtheorem{defi}[theo]{Definition}
\newtheorem{exam}[theo]{Example}
\newtheorem{rema}[theo]{Remark}
\newtheorem{ques}[theo]{Question}

\newcommand{\A }{\mathcal{A}}
\newcommand{\C }{\mathcal{C}}
\newcommand{\B }{\mathcal{B}}
\newcommand{\D }{\mathcal{D}}
\newcommand{\E }{\mathcal{E}}
\newcommand{\FPdim}{\text{FPdim}}
\newcommand{\FQ }{\mathbb{Q}}
\newcommand{\FC }{\mathbb{C}}
\newcommand{\Gal}{\text{Gal}}
\newcommand{\Gr}{\text{Gr}}
\newcommand{\Hom}{\text{Hom}}
\newcommand{\K}{\mathds{k}}
\newcommand{\I }{\mathcal{I}}
\newcommand{\M }{\mathcal{M}}
\newcommand{\Q }{\mathcal{O}}
\newcommand{\rank}{\text{rank}}
\newcommand{\Rep}{\text{Rep}}
\newcommand{\rev}{\text{rev}}
\newcommand{\ssl}{\mathfrak{sl}}
\newcommand{\Y }{\mathcal{Z}}
\newcommand{\Z }{\mathbb{Z}}

\abstract
Let  $\A$ be a  near-group fusion category of type $\Z_3+6$. We show that there  is a modular tensor equivalence $\Y(\A)\cong\C(\Z_3,\eta)\boxtimes\C(\ssl_3,9)_{\Z_3}^0$. Moreover, we  construct two non-trivial  faithful extensions of $\A$ explicitly, whose Drinfeld centers can also be obtained from representation categories quantum groups at root of unity.


\section{Introduction}
A fusion category $\C$ over the complex field $\FC$ is a semisimple finite tensor category. Let $G$ be a finite group. A fusion category $\C$ is called pointed if its simple objects   are all invertible \cite{ENO1}. Equivalently,  $\C$ is tensor equivalent to the category $\text{Vec}_G^\omega$ of finite-dimensional $G$-graded vector spaces over $\FC$, where $\omega\in Z^3(G,\FC^*)$ is a (normalized) $3$-cocycle, which provides the associator of the monoidal structure of $\text{Vec}_G^\omega$ \cite{EGNO,ENO1}.

$\C$ is called a near-group fusion  category if all of the isomorphism classes of its simple objects except one are invertible \cite{Si}.  Specifically, let $n$ be a non-negative integer and let $G$ be the finite group of isomorphism classes of $\C$, and $X$ be the unique non-invertible simple object of $\C$. Then the Grothendieck ring $\Gr(\C)$ is determined by the following fusion relations
\begin{align*}
g\cdot h=gh, g\cdot X=X\cdot g=X,X\cdot X=\sum_{g\in G}g+ nX, \forall g,h\in G.
\end{align*}
We will say $\C$ is a near-group fusion category of type $G+n$. It turns out that the group $G$ and integer $n$ have to meet strict restrictions to guarantee that there exists a near-group fusion category, see \cite{EG,Izumi,Si} and the references therein for details. For example, if $\FPdim(X)$ is irrational, then $n$ is a multiple of the order of $|G|$ \cite[Appendix A]{O}.

Let $G=\Z_3$, it was proved in \cite[Theorem 4.11]{Lar} that if  there is a near-group fusion category of type $\Z_3+n$, then $n=0,2,3,6$. The existences of fusion categories in the first three cases was clear then, and the existence of near-group fusion category of type $\Z_3+6$ was later proved in \cite[Theorem 10.19]{Izumi} \footnote{Indeed, it was pointed by D. Penneys that the existence of this near-group fusion category dates back to 2000, as   the even half of the subfactor whose principal graph was described in Figure 5 \cite[P. 28]{BEK}.}, who attributes the result  to Zhengwei Liu and Noah Snyder. In this paper, we pay our attention to consider the Drinfeld center of near-group fusion category $\A$ of type $\Z_3+6$. We show that $\Y(\A)$ is factored in to Deligne tensor product of a pointed modular fusion category and $\C(\ssl_3,9)_{\Z_3}^0$ (see Theorem \ref{NearGroup}, also \cite[Theorem 5.2]{EdieM}). Notice that both of these fusion categories can be obtained from representation categories of quantum groups of root of unity,  hence the near-group fusion categories of type $\Z_3+6$ is not exotic. And we construct two faithful $G$-graded extensions of $\A$ by using Witt equivalence and non-degenerate fusion categories developed in \cite{DMNO}.

$\mathbf{Note:}$ Recently, the author was  informed  by S-H. Ng and Y. Wang that Theorem \ref{NearGroup} was obtained in \cite[Theorem 5.2]{EdieM} already, the author claims that this is  worked out independently.

This paper is organized as follows. In Section \ref{preliminaries}, we recall some basic notions and notations of  fusion categories, such as  Frobenius-Perron dimension, spherical structure, \'{e}tale algebras, etc.  In Section \ref{mainresults}, we consider the Drinfeld centers and two faithful graded extensions of near-group fusion category $\A$ of type $\Z_3+6$. In Subsection \ref{CenterData}, we first determine the fusion rules of $\C(\ssl_3,9)_{\Z_3}^0$, then we obtain the $S$-matrix with the help of Verlinde formula in Proposition \ref{S-matrix}, in the last, we determine  the structure of $\Y(\A)$ in Theorem \ref{NearGroup}.  In Subsection \ref{extension}, we construct two faithful graded extension of $\A$,  and we determine the fusion relations of them explicitly  in Theorem \ref{Z2extension} and Theorem \ref{Z3extension}, respectively.

\section{Preliminaries}\label{preliminaries}
In this section,  we will recall some most used  definitions and properties about  fusion categories and modular fusion categories, we refer the readers to \cite{DrGNO,EGNO,ENO1,Mu} to standard conclusions for  fusion categories and braided fusion categories.

Let $\C$ be a fusion category,   $\Q(\C)$  the set of isomorphism classes of simple objects of  $\C$. Then there is a unique ring homomorphism $\FPdim(-)$, called the Frobenius-Perron homomorphism,  from the Grothendieck ring $\text{Gr}(\C)$ to $\FC$ such that $\FPdim(X)$ is a positive algebraic integer for all non-zero object $X$ of $\C$ \cite{EGNO,ENO1}. The sum
\begin{align*}
\FPdim(\C):=\sum_{X\in\Q(\C)}\FPdim(X)^2
\end{align*}
is called  the Frobenius-Perron dimension of fusion category $\C$.

A fusion category $\C$  is spherical if $\C$ admits a spherical structure $j$. Specifically, $j$ is a natural isomorphism from identity tensor functor $\text{id}_\C$ to the double dual tensor functor $(-)^{**}$ and $\dim_j(X)=\dim_j(X^*)$ for all objects $X$ of $\C$, where $\dim_j(X)$ is the quantum dimension of an object $X$ is defined by the (categorical) trace of $\text{id}_X$, that is,
\begin{align*}
\dim_j(X)=\text{Tr}(\text{id}_X):=\text{ev}_X\circ(j_X\circ \text{id}_X)\otimes \text{id}_{X^*}\circ\text{coev}_X,
\end{align*}
here we suppress the associativity and   unit constraints  of $\C$.   Then the global (or categorical) dimension of the spherical fusion category $\C$ is defined as
\begin{align*}
\dim(\C):=\sum_{X\in\Q(\C)}\dim_j(X)^2.
\end{align*}
The global dimension $\dim(\C)$ is independent of the choice of the spherical structure of $\C$. In this following, we will fix a spherical structure $j$ on $\C$ and denote $\dim_j(-)$ by $\dim(-)$. A fusion category $\C$ is pseudo-unitary if $\FPdim(\C)=\dim(\C)$.

Let $\C$ be a braided fusion category with  braiding   $c$. Then $\C$ is called a  pre-modular fusion category if $\C$ is spherical. Let $\theta$ be the corresponding ribbon structure of $\C$. The two matrices $S=(S_{X,Y})$ and $T=(\delta_{X,Y}\theta_X)$ are called the $S$-matrix and $T$-matrix of $\C$, respectively, where $S_{X,Y}:=\text{Tr}(c_{Y,X}c_{X,Y})$, $X,Y\in\Q(\C)$.
A modular fusion category $\C$ is a pre-modular fusion category such that the $S$-matrix $S$ is non-degenerate \cite{DrGNO,EGNO}. Drinfeld centers of spherical fusion categories are always modular \cite{Mu}, for example.
It is well-known that a braided pointed fusion category $\text{Vec}_G^\omega$ is modular if and only if $G$ is a metric group equipped with a non-degenerate quadratic form $\eta$ \cite{DrGNO,EGNO}. In this paper, we will use $\C(G,\eta)$ to denote the modular fusion category determined by metric group $(G,\eta)$.

 Recall that a commutative algebra $A$ in a braided fusion category $\C$ is said to be a connected \'{e}tale  algebra if the category $\C_A$ of right $A$-modules in $\C$ is semisimple and   $\Hom_\C(I,A)=\FC$  \cite[Definition 3.1]{DMNO}. For example, let $\Rep(G)\subseteq\C$ be a Tannakian fusion subcategory, then the function (or regular) algebra $\text{Fun}(G)\in\Rep(G)$ is a connected \'{e}tale algebra \cite{DMNO,DrGNO}.

 Let $A$ be a connected \'{e}tale algebra, and let $(M,\mu_M)\in\C_A$, the fusion category of right $A$-modules in $\C$, where $\mu_M:M\otimes A\to M$ is the right $A$-module morphism of $M$. Then $M$ is called a local (or dyslectic) module if $\mu_M=\mu_M\circ (c_{A,M}c_{M,A})$ \cite{DMNO,KiO,Pa}. The category of local modules over   $A$  is a braided fusion category \cite{KiO,Pa}, which will be denoted by $\C_A^0$ below. In addition, if $\C$ is modular, so is $\C_A^0$ \cite{DMNO,DrGNO,KiO}.
\section{Center and extensions of near-group fusion categories}\label{mainresults}

\subsection{Modular data of $\C(\ssl_3,9)_{\Z_3}^0$ and Drinfeld center}\label{CenterData}
In this subsection, we determine the fusion rules  and modular data of $\C(\ssl_3,9)_{\Z_3}^0$, and we show that the Drinfeld center of the near-group fusion category of type $\Z_3+6$ is modular tensor equivalent to  $\C\boxtimes\C(\ssl_3,9)_{\Z_3}^0$, where $\C$ is a pointed modular fusion category of Frobenius-Perron dimension $3$.

Let $k$ be a positive integer. Let $\C(\ssl_3,k)$ be the modular fusion category obtained from semisimplification of the category $\Rep(U_q(\ssl_3))$ of  finite-dimensional representation of $U_q(\ssl_3)$ at root of unity $q=e^\frac{2\pi i}{3(k+3)}$, we refer the readers to \cite{BK,EGNO} for a detailed construction of modular fusion categories $\C(\mathfrak{g},k)$, where $\mathfrak{g}$ is a semisimple Lie algebra over $\FC$.

Let $\alpha_1,\alpha_2$ be the simple roots of $\ssl_3$, and let $\lambda_1=\frac{2\alpha_1+\alpha_2}{3}$ and $\lambda_2=\frac{\alpha_1+2\alpha_2}{3}$ be the fundamental dominant weights of $\ssl_3$ \cite{BK}. Notice that simple objects of $\C(\ssl_3,k)$ are in bijection with pairs of non-negative integers $(m_1,m_2)$ such that $m_1+m_2\leq k$ \cite{BK,Sch1}, each $(m_1,m_2)$ corresponds to a dominant weight $m_1\lambda_1+m_2\lambda_2$, indeed. Hence we will use $(m_1,m_2)$ to denote simple objects of $\C(\ssl_3,k)$. Moreover, for an arbitrary simple object $(m_1,m_2)$, we have
\begin{align}\label{dimformula}
\dim((m_1,m_2))=\frac{\sin\left(\frac{(m_1+1)\pi}{k+3}\right)
\sin\left(\frac{(m_2+1)\pi}{k+3}\right)\sin\left(\frac{(m_1+m_2+2)\pi}{k+3}\right)}
{\sin\left(\frac{2\pi}{k+3}\right)\sin^2\left(\frac{\pi}{k+3}\right)},
\end{align}
which is a positive algebraic integer, so it is the Frobenius-Perron dimension of $(m_1,m_2)$, and then $\C(\ssl_3,k)$ is pseudo-unitary \cite{EGNO}.

If  $3\mid k $, then $\C(\ssl_3,k)$ contains a maximal Tannakian fusion subcategory $\Rep(\Z_3)$, whose isomorphism classes of simple objects are $\{(0,0),(k,0),(0,k)\}$. Let $A:=(0,0)\oplus(k,0)\oplus(0,k)$ be the corresponding connected \'{e}tale algebra of $\Rep(\Z_3)$, and let
\begin{align*}
F_A:\C(\ssl_3,k)\to\C(\ssl_3,k)_A, M\mapsto M\otimes A,
 \end{align*}
 be the free right $A$-module functor, where $\C(\ssl_3,k)_A$ is the fusion category of right $A$-modules in $\C(\ssl_3,k)$ with  unit object $A$.
\begin{lemm}
The isomorphism classes of simple objects of $\C(\ssl_3,9)_{\Z_3}^0$ are
\begin{align*}
\{I,X_1,X_2,X_3,Y_1,Y_2,Y_3,Y_4,Y_5\},
\end{align*}
$\FPdim(X_i)=\FPdim(Y_j)=d$ for $1\leq i,j\leq 3$, $\FPdim(Y_4)=2d+2$ and $\FPdim(Y_5)=2d+1$, where $d:=3+2\sqrt{3}$.
\end{lemm}
\begin{proof}
 It was characterized in \cite[subsection 3.5]{Sch1} that  $\C(\ssl_3,9)_{\Z_3}^0$  contains two types of simple objects. The first classes are  the orbits of simple objects $(m_1,m_2)$  under action of the algebra $A=(0,0)\oplus(9,0)\oplus(0,9)$, where the weight $m_1\lambda_1+m_2\lambda_2$ determined by $(m_1,m_2)$ belongs to root lattice of $\ssl_3$. A direct computation shows that these simple objects are
\begin{align*}
I:&=F_A((0,0))=(0,0)\oplus(0,9)\oplus(9,0),Y_1:=F_A((1,1))=(1,1)\oplus(1,7)\oplus(7,1),\\
Y_2:&=F_A((3,0))=(3,0)\oplus(0,6)\oplus(6,3), Y_3:=F_A((0,3))=(0,3)\oplus(3,6)\oplus(6,0),\\
Y_4:&=F_A((2,2))=(2,2)\oplus(2,5)\oplus(5,2),
Y_5:=F_A((4,1))=(4,1)\oplus(1,4)\oplus(4,4).
\end{align*}
The other type is three simple objects $X_1,X_2,X_3$, which are isomorphic to  $(3,3)$ as objects but  non-isomorphic  as right $A$-modules. The Frobenius-Perron dimensions of simple objects of $\C(\ssl_3,9)_{\Z_3}^0$  follows from \cite[Theorem 1.18]{KiO} and equation (\ref{dimformula}).
\end{proof}
Next, we determine the fusion rules of $\C(\ssl_3,9)_{\Z_3}^0$.
\begin{prop}\label{fusion1}The fusion rules between $Y_j$ ($1\leq j\leq 5$) are the following equations
\begin{align*}
&Y_1\otimes Y_1=   I\oplus Y_1\oplus f_1\oplus Y_4,
Y_1\otimes Y_2=Y_1\oplus Y_2\oplus f_2,Y_1\otimes Y_3=Y_1\oplus Y_3\oplus f_2,\\
&Y_1\otimes Y_4=f_1\oplus 2f_2\oplus f_3,
Y_1\otimes Y_5= Y_2\oplus Y_3\oplus 2f_2\oplus f_3,Y_2\otimes Y_2=2Y_3\oplus f_2,\\
&Y_2\otimes Y_3=I\oplus Y_1\oplus  Y_4\oplus f_3,
Y_2\otimes Y_4=f_1\oplus 2f_2\oplus f_3,\\
&Y_2\otimes Y_5=Y_1\oplus Y_3\oplus 2f_2\oplus f_3,
Y_3\otimes Y_3=2Y_2\oplus f_2,Y_3\otimes Y_4=f_1\oplus 2f_2\oplus f_3,\\
&Y_3\otimes Y_5=Y_1\oplus Y_2\oplus 2f_2\oplus f_3, Y_4\otimes Y_4=I\oplus 2f_1\oplus 5f_2\oplus 2f_3,\\
&Y_4\otimes Y_5=2f_1\oplus Y_4\oplus 4f_2\oplus 2f_3,Y_5\otimes Y_5=I\oplus2f_1\oplus 4f_2\oplus 2f_3,
\end{align*}
where $f_1:=Y_1\oplus Y_2\oplus Y_3$, $f_2:=Y_4\oplus Y_5$ and $f_3:=X_1\oplus X_2\oplus X_3$.
\end{prop}
\begin{proof}
We compute these fusion rules by using program SageMath \cite{SageMath}. For example,
\begin{align*}
Y_1\otimes Y_1&=F_A((1,1))\otimes_A F_A((1,1))=F_A((1,1)\otimes (1,1))\\
&=F_A((0,0)\oplus(0,3)\oplus(3,0)\oplus2(1,1)\oplus(2,2))\\
&=F_A((0,0))\oplus F_A((0,3))\oplus F_A((3,0))\oplus 2F_A((1,1))\oplus F_A((2,2))\\
&=I\oplus 2Y_1\oplus Y_2\oplus Y_3\oplus Y_4.
\end{align*}
The rest fusion rules can be computed in the same way.
\end{proof}
\begin{prop}\label{fusion2}For simple objects $X_i$ ($1\leq i\leq 3$) and $Y_l$ ($1\leq l\leq 5$), the fusion rules are
\begin{align*}
&X_i\otimes Y_1=f_2\oplus X_j\oplus X_k, X_i\otimes Y_2=Y_2\oplus f_2\oplus X_i,X_i\otimes Y_3=Y_3\oplus f_2\oplus X_i,\\
&X_i\otimes Y_4=f_1\oplus 2f_2\oplus f_3,X_i\otimes Y_5=f_1\oplus 2f_2\oplus X_j\oplus X_k,\\
&X_i\otimes X_i=I\oplus Y_2\oplus Y_3\oplus Y_4\oplus 2X_i,
X_i\otimes X_j=Y_1\oplus f_2\oplus X_k.
\end{align*}
where $1\leq i,j,k\leq3$ are distinct and $f_2=Y_4\oplus Y_5$ .
\end{prop}
\begin{proof}Notice that simple objects  $X_j$ are self-dual, $1\leq j\leq 3$. For $1\leq i\neq j\leq3$, by decomposing $\FPdim(X_i\otimes X_j)$ into a sum of Frobenius-Perron dimensions of simple objects of $\C(\ssl_3,9)_{\Z_3}^0$, and together with Proposition \ref{fusion1}, we find \begin{align*}
X_i\otimes Y_5&=Y_1\oplus Y_2\oplus Y_3\oplus 2Y_4\oplus 2Y_5\oplus X_j\oplus X_k,\\
X_i\otimes Y_4&=Y_1\oplus Y_2\oplus Y_3\oplus 2Y_4\oplus 2Y_5\oplus X_1\oplus X_2\oplus X_3,
\end{align*}
where $i,j,k$ are distinct.
By Proposition \ref{fusion1},  $X_1\otimes Y_2=Y_2\otimes Y_4\oplus Y_5\oplus X_j$ for some $j$, let $X_j\otimes Y_2=Y_2\oplus Y_4\oplus Y_5\oplus X_k$. Since
$(X_1\otimes Y_2)\otimes Y_2=X_1\otimes (Y_2\otimes Y_2)$, we find $X_j=X_k=X_1$. By symmetry of this computation, we have $X_j\otimes Y_2=Y_2\oplus Y_4\oplus Y_5\oplus X_j$ for $1\leq j\leq 3$. Similarly, for $1\leq i\leq 3$, it follows the decomposition of $(X_i\otimes Y_1)\otimes Y_1=X_i\otimes(Y_1\otimes Y_1)$ that $X_i\otimes Y_1=Y_4\oplus Y_5\oplus X_j\oplus X_k$ with $i,j,k$ being distinct.

By Proposition \ref{fusion1}, for $1\leq i\neq j\leq3$, we know
$X_i\otimes X_j=Y_1\oplus Y_4\oplus Y_5\oplus X_{ij}$
for some $X_{ij}\in\{X_1,X_2,X_3\}$, so
$\dim_\FC\Hom_{\C(\ssl_3,9)_{\Z_3}^0}(X_i\otimes X_j,X_i\otimes X_j)=4$. Meanwhile,
\begin{align*}
\dim_\FC\Hom_{\C(\ssl_3,9)_{\Z_3}^0}(X_i\otimes X_j,X_i\otimes X_j)=\dim_\FC\Hom_{\C(\ssl_3,9)_{\Z_3}^0}(X_i\otimes X_i,X_j\otimes X_j),
\end{align*}
hence $X_i\otimes X_i=I\oplus Y_2\oplus Y_3\oplus Y_4\oplus 2 X_{ii}$ for some $X_{ii}\in\{X_1,X_2,X_3\}$, and $X_{ii}\neq X_{jj}$ if $i\neq j$. In summary, we have
\begin{align*}
X_i\otimes X_j=\left\{
                 \begin{array}{ll}
                   Y_1\oplus Y_4\oplus Y_5\oplus X_{ij}, & \hbox{ if $i\neq j$;} \\
                   I\oplus Y_2\oplus Y_3\oplus 2X_{ii}, & \hbox{if $i=j$.}
                 \end{array}
               \right.
\end{align*}
If $X_i\neq X_{ii}$, then it follows from the decomposition of $(X_i\otimes X_i)\otimes Y_1=X_i\otimes (X_i\otimes Y_1)$ and the decomposition of $X_i\otimes Y_1$ that
\begin{align*}
4=\dim_\FC\Hom_{\C(\ssl_3,9)_{\Z_3}^0}(X_i\otimes Y_1,X_{ii}\otimes Y_1)=\dim_\FC\Hom_{\C(\ssl_3,9)_{\Z_3}^0}(X_i\otimes X_{ii},Y_1\otimes Y_1)=3,
\end{align*}
it is a contradiction. Therefore, $X_i=X_{ii}$ for all $i$; consequently, $X_i\otimes X_j=Y_1\oplus Y_4\oplus Y_5\oplus X_k$ for $i,j,k$ being distinct.
\end{proof}

In the last, we determine the modular data of $\C(\ssl_3,9)_{\Z_3}^0$. Let $\zeta_n:=e^\frac{2\pi i}{n}$.
\begin{lemm}\label{ribbon1}
The ribbon structures of simple objects of $\C(\ssl_3,9)_{\Z_3}^0$ are given by
\begin{align*}
\theta_{Y_1}=\theta_{X_i}=\zeta_4,\theta_{Y_2}=\theta_{Y_3}=-1,\theta_{Y_4}=\zeta_3^2,
\theta_{Y_5}=1, 1\leq i\leq 3,
\end{align*}
\end{lemm}
\begin{proof}
It follows from \cite{BK} that the ribbon structure of simple object $(m_1,m_2)$ of $\C(\ssl_3,k)$ is
\begin{align*}
\theta_{(m_1,m_2)}=e^{\frac{(m_1^2+3m_1+m_1m_2+3m_2+m_2^2)}{3(k+3)}2\pi i},
\end{align*}
meanwhile, \cite[Theorem 1.17]{KiO} states  that ribbon structures of objects of $\C(\ssl_3,9)_{\Z_3}^0$ is inherited from $\C(\ssl_3,9)$, then the lemma follows.
\end{proof}
Given a modular fusion category $\C$, the balancing equation \cite[Proposition 8.13.8]{EGNO} states that
\begin{align*}
S_{X,Y}=\theta_X^{-1}\theta_Y^{-1}\sum_{Z\in\Q(\C)}N_{X,Y}^Z\dim(Z)\theta_Z,\quad \forall X,Y\in\Q(\C),
 \end{align*}
where $N_{X,Y}^Z:=\dim_\FC\dim_\C(X\otimes Y,Z)$. A direct computation with the balancing equation, we  obtain the $S$-matrix of $\C(\ssl_3,9)_{\Z_3}^0$.
\begin{prop}\label{S-matrix}The $S$-matrix of $\C(\ssl_3,9)_{\Z_3}^0$ is
      \begin{align*}
      \left(
        \begin{array}{ccccccccc}
          1 & d   & d              & d              & 2(1+d) & 1+2d & d & d & d \\
          d & 3d  & d              & d              & 0      & -d        &-d & -d & -d \\
          d & d   & -(1+2\zeta_4)d & -(1-2\zeta_4)d & 0      & -d        & d & d& d\\
          d & d   & -(1-2\zeta_4)d & -(1+2\zeta_4)d & 0      & -d         & d & d & d \\
     2(1+d) & 0   & 0              & 0              & -2(1+d)& 2(1+d)         & 0 & 0 & 0 \\
       1+2d & -d  & -d             & -d             & 2(1+d) & 1        & -d & -d & -d \\
          d & -d  & d              & d              & 0      & -d         &3d & -d & -d \\
          d & -d  & d              & d              & 0      & -d         & -d & 3d& -d\\
          d & -d  & d              & d              & 0      & -d         & -d & -d & 3d \\
        \end{array}
      \right),
      \end{align*}
      where $d=3+2\sqrt{3}$.
\end{prop}

We are ready to prove the following theorem
\begin{theo}(\cite[Theorem 5.2]{EdieM}) \label{NearGroup}There is a modular tensor equivalence
$\C(\Z_3,\eta)\boxtimes\C(\ssl_3,9)_{\Z_3}^0\cong\Y(\A)$, where  $\A$ is a near-group fusion category of type $\Z_3+6$.
\end{theo}
\begin{proof}
It was shown in \cite{Ga,Sch1} that $\C(\ssl_3,9)_{\Z_3}^0$ contains a non-trivial \'{e}tale algebra $A$ such that
\begin{align*}
\C(\Z_3,\eta)\cong(\C(\ssl_3,9)_{\Z_3}^0)_A^0
 \end{align*}
is a modular tensor equivalence, where $\Z_3=\langle g\rangle$ and $\eta(g)=\zeta_3^2$. Hence, it follows from \cite[Corollary 3.30]{DMNO} that there is a modular tensor equivalence
\begin{align*}
\C(\Z_3,\eta)^\rev\boxtimes\C(\ssl_3,9)_{\Z_3}^0
=\C(\Z_3,\eta^{-1})\boxtimes\C(\ssl_3,9)_{\Z_3}^0\cong\Y((\C(\ssl_3,9)_{\Z_3}^0)_A).
\end{align*}
Let $\A:=(\C(\ssl_3,9)_{\Z_3}^0)_A$. Then
\begin{align*}
\FPdim(\A)^2=3\FPdim(\C(\ssl_3,9)_{\Z_3}^0)=12^2(2+\sqrt{3})^2.
\end{align*}
Notice that connected
\'{e}tale subalgebras $B$ of $\I(I)$ are in bijective correspondence with fusion subcategories $\B$ of  $\A$ \cite[Theorem 4.10]{DMNO} satisfying $\FPdim(\B)\FPdim(B)=\FPdim(\A)$, where $\I:\A\to\Y(\A)$ is the adjoint  functor to the forgetful functor $F:\Y(\A)\to \A$. Since  \'{e}tale subalgebras of $\I(I)$ have trivial ribbon structures \cite[Theorem 4.1]{NS}, the image of $\C(\Z_3,\eta^{-1})$ is a non-trivial pointed fusion subcategory of $\A$. Up to isomorphism, a direct computation shows that $\C(\Z_3,\eta^{-1})\boxtimes\C(\ssl_3,9)_{\Z_3}^0$ contains exactly four such simple objects
  \begin{align*}
  I\boxtimes I, I\boxtimes Y_5, g\boxtimes Y_4, g^2\boxtimes Y_4.
  \end{align*}
  By comparing their Frobenius-Perron dimensions, we find that there exist only two possible non-trivial \'{e}tale subalgebras of $\I(I)$: $I\boxtimes I\oplus I\boxtimes Y_5$ and $\I(I)$ itself. Therefore, $\A$ only contains a unique non-trivial fusion subcategory, which is also pointed.

By decomposing $\FPdim(\A)=24+12\sqrt{3}$ into sums of squares of  Frobenius-Perron dimensions of simple objects of $\A$ over the field $\FQ(\sqrt{3})$,  there exist only two solutions:
\begin{align*}
24+12\sqrt{3}=3+(3+2\sqrt{3})^2=3+3(2+\sqrt{3})^2.
\end{align*}
In the later case, $\A$ contains exactly three simple objects $\{Z_1,Z_2,Z_3\}$ (up to isomorphism) of Frobenius-Perron dimension $2+\sqrt{3}$,  so $Z_i\otimes Z_i^*$ is a direct sum of simple objects of $\A$. However, it is easy to see that there does not exist such a decomposition. Hence, $\Q(\A)$ contains a unique non-invertible object $X$ of Frobenius-Perron dimension $3+\sqrt{3}$, thus we have
\begin{align*}
X\otimes X=I\oplus g\oplus g^2\oplus 6X.
\end{align*}
That is, $\A$ is a near-group fusion category of type $\Z_3+6$.
\end{proof}

It would be interesting to show whether other near-group fusion categories (generally,  quadratic fusion categories) are coming from quantum groups at roots of unity. For example, by using the same method, we also have the following result.
 \begin{prop}Let $\C$ be a near-group fusion category of type $\Z_4+4$, then there exists a modular tensor equivalence  $\Y(\C)_{\Z_2}^0\cong\C(\Z_2,\eta)\boxtimes\C(\ssl_3,5)_\text{ad}$.
 \end{prop}
\begin{proof}Notice that the simple objects of $\C(\ssl_3,5)_\text{ad}$ are \begin{align*}
 \{(0,0),(0,3),(3,0),(1,1),(2,2),(1,4),(4,1)\}.
 \end{align*}
It is well-known that $\C(\ssl_3,5)_\text{ad}$ contains a connected \'{e}tale algebra $A=(0,0)\oplus(2,2)$ such that $\C(\Z_2,\eta)=(\C(\ssl_3,5)_\text{ad})_A^0$ \cite{Ga,Sch1}. Hence,  $\C(\Z_2,\eta)^\text{rev}\boxtimes\C(\ssl_3,5)_\text{ad}\cong\Y(\B)$ for  fusion category $\B:=(\C(\ssl_3,5)_\text{ad})_A$ \cite{DMNO}. Then $\B$ is a fusion category of rank $4$ and fusion rules are
 \begin{align*}
 X\otimes X^*=I\oplus X\oplus X^*, g\otimes X=X^*=X\otimes g, \\
 g\otimes g=I,X\otimes X=X^*\otimes X^*=g\oplus X\oplus X^*.
 \end{align*}
 Indeed, let $X=F(I\boxtimes(1,4))$, where $F:\Y(\B)\to\B$ is the forgetful functor, then $X$ is a simple object of Frobenius-Perron dimension $1+\sqrt{2}$, hence
 \begin{align*}
 X\otimes X&=F(I\boxtimes(1,4)\otimes I\boxtimes(1,4))\\
 &=F(I\boxtimes(3,0)\oplus I\boxtimes(4,1))\\
 &=F(I\boxtimes(3,0))\oplus F(I\boxtimes(4,1)).
 \end{align*}
 Since $\FPdim((3,0))=2+\sqrt{2}$,  it is not simple by counting the Frobenius-Perron dimension of $\B$. Hence, $F(I\boxtimes(3,0))$ must contains an invertible object $g$ as a direct summand. However, $\theta_{(3,0)}=\zeta_4^3$ by Lemma \ref{ribbon1}, $I\boxtimes(3,0)\nsubseteq\I(I)$, where $\I$ is the adjoint functor to $F$, so $g\neq I$. Thus $g\subseteq F(I\boxtimes(3,0))$ and $X^*=F(I\boxtimes(4,1))\subseteq X\otimes X^*$, which implies
 \begin{align*}
 X\otimes X=g\oplus X\oplus X^*=X^*\otimes X^*,
 \end{align*}
 then the rest fusion rules are easy. And it was proved that the $\Z_2$-equivariantization of $\B$ is a near-group fusion category of type $\Z_4+4$ \cite{LMP}\cite{Izumi}.
\end{proof}

\subsection{Faithful extensions of $\A$}\label{extension}

In this subsection, we construct two faithful graded extension of the near-group fusion categories of type $\Z_3+6$ by using Witt equivalences of non-degenerate fusion categories \cite{DMNO}.
\begin{theo}\label{Z2extension}
The near-group fusion category $\A$ admits a faithful $\Z_2$-extension $\B$, which  is a fusion category of rank $8$. Explicitly, $\Q(\B)=\{I,g,g^2,X,Y,Z_1,Z_2,Z_3\}$, and
\begin{align*}
&X\otimes X=I\oplus g\oplus g^2\oplus 6X,g\otimes X=X\otimes g=X,X\otimes Y=Y\otimes X=Z_1\oplus Z_2\oplus Z_3,\\
& g\otimes Z_3=Z_1=Z_3\otimes g^2,g^2\otimes Z_3=Z_2=Z_3\otimes g,Z_i\otimes Z_j=g^{i+2j}\oplus 2X,\\
&Y\otimes Y=I\oplus g\oplus g^2,Y\otimes Z_i=Z_i\otimes Y=X, Z_i\otimes X=X\otimes Z_i=Y\oplus 2Z_1\oplus 2Z_2\oplus 2Z_3.
\end{align*}
\end{theo}
\begin{proof}
Consider the modular fusion category $\C(\ssl_3,9)_{\Z_3}^0\boxtimes\C(\ssl_2,4)$. Since $\C(\ssl_2,4)$ contains a Tannakian fusion subcategory $\Rep(\Z_2)$ such that $\C(\ssl_2,4)_{\Z_2}^0\cong\C(\Z_3,\eta^{-1})$, we have the following braided tensor equivalences
\begin{align*}
\left(\C(\ssl_3,9)_{\Z_3}^0\boxtimes\C(\ssl_2,4)\right)_{\Z_2}^0\cong
\C(\ssl_3,9)_{\Z_3}^0\boxtimes\C(\ssl_2,4)_{\Z_2}^0\cong
\Y(\A),
\end{align*}
by Theorem \ref{NearGroup}. Then it follows from \cite[Theorem 1.3]{ENO2}  that $\C(\ssl_3,9)_{\Z_3}^0\boxtimes\C(\ssl_2,4)\cong\Y(\B)$ for a fusion category $\B$ with $\B$ being a faithful $\Z_2$-extension of $\A$. Let $\B=\oplus_{h\in\Z_2}\B_g$ be the corresponding grading of $\B$ with $\B_e=\A$.

Let $F:\Y(\B)\to\B$ be the forgetful functor. Notice that the above tensor equivalence implies that $\Y(\B)$ contains a simple object of Frobenius-Perron dimension $\sqrt{3}$, the image of this simple object in $\B$, denoted by $Y$, must be simple. Obviously,  $Y\in\B_h$ and $Y\otimes Y=I\oplus g\oplus g^2$. Since
\begin{align*}
\Hom_\B(X\otimes Y,X\otimes Y)\cong\Hom_\B(X,X\otimes Y\otimes Y)=\FC^3,
\end{align*}
$X\otimes Y=Z_1\oplus Z_2\oplus Z_3$, where $Z_j$ are distinct simple objects. It follows from
\begin{align*}
\Hom_\B(Y\otimes Y, X)=\Hom(g^i\otimes Y,Z_j)\cong\Hom(g^i,Z_j\otimes Y)=0
\end{align*}
that $Z_j\otimes Y=n_jX$  for some integers $n_j\geq1$, $1\leq j\leq 3$.  Therefore,
\begin{align*}
3X=X\otimes (Y\otimes Y)=(X\otimes Y)\otimes Y=(n_1+n_2+n_3)X,
\end{align*}
thus $Z_j\otimes Y=X$, similarly, $Y\otimes Z_j=X$. By comparing the Frobenius-Perron dimensions, we obtain $\Q(\B_h)=\{Y,Z_1,Z_2,Z_3\}$.

Notice that at least one of $Z_j$ are self-dual. Without loss of generality, assume $Z_3\otimes Z_3=I\oplus 2X$,  $g\otimes Z_3=Z_1=Z_3\otimes g^k$ and $g\otimes Z_1=Z_2$ for some $k$. Then $X\otimes Z_i=X\otimes Z_3=Y\oplus 2Z_1\oplus 2Z_2\oplus 2Z_3$ and
$Z_i\otimes Z_j=g^{i+kj}\otimes 2X$. Meanwhile, ribbon structure of simple objects of $\C(\ssl_2,4)$ are explicitly $1,1,\zeta_3,\zeta_8,\zeta_8^5$, so $\C(\ssl_2,4)\boxtimes\C(\ssl_3,9)_{\Z_3}^0$ has exactly five simple objects of trivial twists by Lemma \ref{ribbon1}, then it follows from \cite[Corollary 2.16]{O} that $\I(I)$ can't be multiplicity-free. Hence, $\Gr(\B)$ is noncommutative, so $g\otimes Z_3=Z_3\otimes g^2$.
\end{proof}

Notice that we also have modular equivalence
\begin{align*}
(\C(\Z_3,\eta^{-1})\boxtimes\C(\ssl_3,9))_{\Z_3}^0\cong\C(\Z_3,\eta^{-1})
\boxtimes\C(\ssl_3,9))_{\Z_3}^0
\cong\Y(\A)
\end{align*}
hence, there  $\C(\Z_3,\eta)\boxtimes\C(\ssl_3,9)\cong\Y(\D)$ as modular fusion category by \cite[Theorem 1.3]{ENO2}, where $\D$ is a faithful $\Z_3$-extension of $\A$.

\begin{lemm}\label{dimension}
Fusion category $\D$ contains a simple object of Frobenius-Perron dimension $1+\sqrt{3}$.
\end{lemm}
\begin{proof}
 By considering the Frobenius-Perron dimensions of simple objects of $\C(\ssl_3,9)$, we know it contains a simple object $W$ (the simple object $(1,0)$, for example) with non-trivial twist of Frobenius-Perron dimension $1+\sqrt{3}$. Let $F:\Y(\D)\to\D$ be the forgetful functor, we claim that $F(I\boxtimes W)$ must be simple
 Indeed, if   $F(I\boxtimes W)$ is not simple, then $F(I\boxtimes W)=W_1\oplus W_2$, where $W_j$ are simple. Then $\FPdim(W_j)=2\cos\frac{\pi}{n_j}$ for some integer $n_j\geq3$ \cite{EGNO}, which means $n_1=3$ and $n_2=6$. Therefore, each grading component of $\C$ contains at least three invertible simple objects up to isomorphism and the non-trivial components contains simple objects of Frobenius-Perron dimension $\sqrt{3}$. Let $\B$ be the strictly weakly integral fusion subcategory of $\D$, so $\FPdim(\B_\text{pt})=3+6a$ and $\FPdim(\D)=3+6a+6b$. Notice that
 \begin{align*}
 \FPdim(\B_\text{pt})\mid \FPdim(\B)\mid\FPdim(\D)=36(2+\sqrt{3})
  \end{align*}
  by \cite[Proposition 8.15]{ENO1}, which then means $a=1$ and $b=0$, it is impossible. Hence object $F(I\boxtimes W)$ must be simple.
\end{proof}

\begin{theo}\label{Z3extension}
The faithful $\Z_3$-extension $\D$ is a fusion category of rank $12$, which has commutative fusion rules. Explicitly, $\Q(\D)=\{I,g,g^2,X,V,W_1,W_2,W_3,V^*,W_1^*,W_2^*,W_3^*\}$, and
\begin{align*}
&X\otimes X=I\oplus g\oplus g^2\oplus 6X,g\otimes X=X,g\otimes W_1=W_2,g\otimes W_2=W_3,V^*\otimes W_j=2X,\\
&W_i\otimes W_j=g^{i+j-2}\otimes W_1^*\oplus V^*,V\otimes V^*=I\oplus g\oplus g^2\oplus 3X,W_i\otimes W_j^*=g^{i-j}\oplus X,\\
&V\otimes W_i=W_1^*\oplus W_2^*\oplus W_3^*\oplus V^*,V\otimes V=W_1^*\oplus W_2^*\oplus W_3^*\oplus 3V^*,\\
&W_i\otimes X=W_1\oplus W_2\oplus W_3\oplus 2V, X\otimes V=V\otimes X=2W_1\oplus 2W_2\oplus 2W_3\oplus 3V.
\end{align*}
\end{theo}
\begin{proof}Let $\D=\oplus_{h\in\Z_3}\D_h$ with $\D_e=\A$. To simplify notations, we use $\C(\Z_3,\eta^{-1})\boxtimes\C(\ssl_3,9)$ to denote the Drinfeld center $\Y(\D)$. By Lemma \ref{dimension}, both of the non-trivial components  $\D_h$ and $\D_{h^2}$ contain at least three  simple objects $W_i$ of Frobenius-Perron dimension $1+\sqrt{3} $, and  obviously $W_i\otimes W_i^*=I\oplus X$. In fact, we can regard $F(I\boxtimes(0,1))$ as one of the simple objects $W_j$. Assume $g\otimes W_1=W_2$ and $g\otimes W_2=W_3$. Hence  $W_i\otimes X=X\otimes W_i=W_1\oplus W_2\oplus W_3\oplus 2V$ with $\FPdim(V)=3+\sqrt{3}$. Indeed, since
\begin{align*}
\dim_\FC\Hom(X\otimes W_i,X\otimes W_i)=\dim_\FC\Hom(X,X\otimes W_i\otimes W_i)=7,
\end{align*}
$W_i\otimes X=W_1\oplus W_2\oplus W_3\oplus M$ with $M$ not being simple. If $M$ is multiplicity-free, then $M=M_1\oplus M_2\oplus M_3\oplus M_4$, where $M_j$ ($1\leq j\leq4$) are simple. Let $\FPdim(M_4)$ being the smallest among $\FPdim(W_j)$, so $\FPdim(M_4)\leq\frac{3+\sqrt{3}}{2}$. Meanwhile, $N_{X,W_i}^{M_4}=N_{W_i,M_4^*}^X=1$,  thus $\FPdim(W_i)\FPdim(M_4)\geq\FPdim(X)$, which means $\FPdim(M_4)\geq\frac{3+\sqrt{3}}{2}$. However, $\frac{3+\sqrt{3}}{2}$ is not an algebraic integer, it is a contradiction. So $M=2V$ with $V$ being simple. By comparing the Frobenius-Perron dimensions of simple objects, we have $\Q(\D_h)=\{W_1,W_2,W_3,V\}$  and $\Q(\D_{h^2})=\{W_1^*,W_2^*,W_3^*,V^*\}$, so   $\text{rank}(\D)=12$.

 Obviously, $V\otimes V^*=I\oplus g\oplus g^2\oplus 3X$.
  Since
\begin{align*}
\dim_\FC\Hom(V\otimes W_i,V\otimes W_i)=\dim_\FC\Hom(V,V\otimes (I\oplus X))=4,
\end{align*}
we also have $V\otimes V=W_1^*\oplus W_2^*\oplus W_3^*\oplus 3V^*$ and $V\otimes W_i=W_i\otimes V=W_1^*\oplus W_2^*\oplus W_3^*\oplus V^*$. Therefore,
$X\otimes V=V\otimes X=2W_1\oplus 2W_2\oplus 2W_3\oplus 3V$, $V^*\otimes W_j=W_j^*\otimes V=2X$.

Notice that $\text{Vec}_{\Z_3}$ is exactly the  image of $\C(\Z_3,\eta^{-1})$ under the forgetful functor $F$, hence $g\otimes W_i=W_i\otimes g$ for all $1\leq i\leq3$. Consequently, we know $\D$ has commutative fusion rules.  Assume $W_1\otimes W_1=V^*\oplus W_l^*$ for some simple object $W_l$, then
\begin{align*}
W_i\otimes W_j=W_1\otimes W_1\otimes g^{i+j-2}=V^*\oplus W_l^*\otimes g^{i+j-2}
 \end{align*}
Thus, we only need to determine $W_1\otimes W_1$. Meanwhile,
\begin{align*}
F(I\boxtimes(0,1))\otimes F(I\boxtimes(0,1))=F(I\boxtimes (0,1)\otimes (0,1))=F(I\boxtimes(1,0)\oplus I\boxtimes(0,2)),
\end{align*}
which implies $W_j^*\subseteq W_j\otimes W_j$ for all $1\leq j\leq 3$, as $I\boxtimes(1,0)$ is the dual object of $I\boxtimes(0,1)$. This finishes the proof of the theorem.
\end{proof}

  A direct computations shows that the following simple objects
\begin{align*}
&I\boxtimes (0,0),I\boxtimes(0,9), I\boxtimes(0,9),I\boxtimes (4,1),I\boxtimes (4,4),I\boxtimes (1,4),\\
&g\boxtimes(2,2), g\boxtimes(2,5), g\boxtimes(5,2),
g^2\boxtimes(2,2), g^2\boxtimes(2,5), g^2\boxtimes(5,2)
\end{align*}
are the simple objects of $\C(\Z_3,\eta^{-1})\boxtimes\C(\ssl_3,9)$ with trivial ribbon.
Since $\text{Vec}_{\Z_3}\subseteq\A\subseteq\D$, we know $\I(I)$ contains two non-trivial connected \'{e}tale subalgebras $A_1\subseteq A_2$ such that $\Y(\D)_{A_1}^0\cong\Y(\A)$ and $\Y(\D)_{A_2}^0\cong\Y(\text{Vec}_{\Z_3})$ \cite{DMNO}, where $A_1:=\text{Fun}(\Z_3)=I\boxtimes (0,0)\oplus I\boxtimes(0,9)\oplus I\boxtimes(0,9)$. Notice that $A_2$ is an \'{e}tale algebra over $A_1$ and $\FPdim(A_2)=24+12\sqrt{3}$, we must have \begin{align*}
A_2=A_1\oplus I\boxtimes(4,1)\oplus I\boxtimes (4,4)\oplus I\boxtimes (1,4).
 \end{align*}
And $\text{Gr}(\D)$ is commutative, so $\I(I)$ is multiplicity-free by
\cite[Corollary 2.16]{O}. Then
\begin{align*}
\I(I)=A_2\oplus \oplus_{j=1}^2 (g^j\boxtimes(2,2)\oplus g^j\boxtimes(2,5)\oplus g^j\boxtimes(5,2)).
\end{align*}

\begin{rema}In addition, $\C(\ssl_2,4)\boxtimes\C(\ssl_3,9)\cong\Y(\C)$, where $\C$ is a $\Z_6$-extension of $\A$, $\C$ is also a $\Z_2$-extension of $\D$ and $\Z_3$-extension of $\B$. Notice that
\begin{align*}
\Hom(W_i\otimes Y,W_i\otimes Y)=\Hom(W_i, W_i\otimes Y\otimes Y)=\FC,
\end{align*}
so $U:=W_i\otimes Y=W_1\otimes Y$ must be simple, and $U\otimes U^*=I\oplus g\oplus g^2\oplus3X$. Notice that
\begin{align*}
\Hom_\C(V\otimes Y, V\otimes Y)=\Hom_\C(V,V\otimes Y\otimes Y)=\FC^3,
\end{align*}
Hence $V\otimes Y=T_1\oplus T_2\oplus T_3$, where $T_j$ are simple. Assume $\FPdim(T_3)\leq\FPdim(T_j)$ for all $j$, then $\FPdim(T_3)\leq1+\sqrt{3}$, meanwhile, $N_{V,Y}^{T_3}=N_{Y,T_3^*}^{V^*}=1$, so
\begin{align*}
\FPdim(Y)\FPdim(T_3)\geq\FPdim(V)=3+\sqrt{3}.
 \end{align*}
Therefore, $\FPdim(T_j)=1+\sqrt{3}$ for all $1\leq j\leq 3$.

Consequently,
$\C=\oplus_{h\in\Z_6}\C_h$ with $\rank(\C)=24$, and
\begin{align*}
\Q(\C_e)&=\{I,g,g^2,X\},\Q(\C_h)=\{U^*,T_1^*,T_2^*,T_3^*\},
\Q(\C_{h^2})=\{W_1,W_2,W_3,V\},\\
\Q(\C_{h^3})&=\{Z_1,Z_2,Z_3,Y\},\Q(\C_{h^4})=\{W_1^*,W_2^*,W_3^*,V^*\},
\Q(\C_{h^5})=\{U,T_1,T_2,T_3\}.
\end{align*}
The fusion rules can be obtained  by in a similar method, we omit it here.
\end{rema}
\section*{Acknowledgements}
The author thanks D. Penneys for clarification on the existence of near-group fusion category of type $\Z_3+6$ and for providing reference \cite{BEK}. The author thanks C. Edie-Michell, S-H. Ng and Y. Wang for comments and communications. The author  is  supported by the National Natural Science Foundation of China (no.12101541), the Natural Science Foundation of Jiangsu Province (no.BK20210785), and the Natural Science Foundation of Jiangsu Higher Institutions of China (no.21KJB110006).

\bigskip\author{Zhiqiang Yu\\ \thanks{Email:\,zhiqyumath@yzu.edu.cn}\\{\small School  of Mathematical Science,  Yangzhou University, Yangzhou 225002, China}}

\end{document}